\documentclass[11pt]{amsart}
\usepackage{amsmath,amssymb,amsthm,mathrsfs}
\usepackage[all]{xy}
\usepackage{color}

\newtheorem{Def}{Definition}[section]
\newtheorem{Thm}[Def]{Theorem}
\newtheorem{Prop}[Def]{Proposition}
\newtheorem{Cor}[Def]{Corollary}
\newtheorem{Ex}[Def]{Example}
\newtheorem{Lem}[Def]{Lemma}
\newtheorem{Conj}[Def]{Conjecture}
\newtheorem{Rem}[Def]{Remark}

\newcommand{\C}{\mathbb{C}}
\newcommand{\R}{\mathbb{R}}

\newcommand{\HH}{\mathbb{H}}
\newcommand{\Z}{\mathbb{Z}}
\newcommand{\N}{\mathbb{N}}
\newcommand{\DD}{\mathcal{D}}

\newcommand{\PP}{\mathcal{P}}
\newcommand{\RR}{\mathcal{R}}

\newcommand{\ZZ}{\mathcal{Z}}
\newcommand{\Stab}{\mathrm{Stab}}
\newcommand{\NS}{\mathrm{NS}}
\newcommand{\Hom}{\mathrm{Hom}}

\newcommand{\GL}{\mathrm{GL}}
\newcommand{\Aut}{\mathrm{Aut}}
\newcommand{\ch}{\mathrm{ch}}

\newcommand{\rank}{\mathrm{rank}}
\newcommand{\tr}{\mathrm{tr}}
\newcommand{\tf}{\mathrm{tf}}

\newcommand{\NN}{\mathcal{N}}
\newcommand{\WP}{\mathrm{WP}}
\newcommand{\Ber}{\mathrm{Ber}}
\newcommand{\cpx}{\mathrm{cpx}}
\newcommand{\Muk}{\mathrm{Muk}}
\newcommand{\Td}{\mathrm{Td}}
\newcommand{\CY}{\mathrm{CY}}

\begin{document}
\title[Weil--Petersson geometry on Stab]{Weil--Petersson geometry on the space of Bridgeland stability conditions}
\author{Yu-Wei Fan \ \ \ Atsushi Kanazawa \ \ \ Shing-Tung Yau}
\date{}

\maketitle

\begin{abstract}
We investigate some differential geometric aspects of the space of Bridgeland stability conditions on a Calabi--Yau triangulated category 
The aim is to give a provisional definition of Weil--Petersson geometry on the stringy K\"ahler moduli space. 
We study in detail a few basic examples to support our proposal.  
In particular, we identify our Weil--Petersson metric with the Bergman metric on a Siegel modular variety in the case of the self-product of an elliptic curve. 
\end{abstract}


\section{Introduction}
The present article investigates some differential geometric aspects of the space of Bridgeland stability conditions on a Calabi--Yau triangulated category.  

The motivation of our study comes from mirror symmetry. 
It is classically known that the complex moduli space $\mathcal{M}_{\mathrm{cpx}}(Y)$ of a projective Calabi--Yau manifold $Y$ comes equipped with a canonical K\"ahler metric, called the Weil--Petersson metric. 
The existence of such a natural metric often implies strong results that one cannot obtain by purely algebraic methods. 
In the case of Calabi--Yau threefold, this metric provides a fundamental differential geometric tool, known as the special K\"ahler geometry, to study mirror symmetry. 
In light of duality between complex geometry and K\"ahler geometry of a mirror pair of Calabi--Yau manifolds, a natural problem is to construct the mirror object of the Weil--Petersson geometry,
which should be defined on the stringy K\"ahler moduli space $\mathcal{M}_{\mathrm{Kah}}(X)$ of a mirror Calabi--Yau manifold $X$ of $Y$. 
However, there is yet no mathematical definition of $\mathcal{M}_{\mathrm{Kah}}(X)$ at present. 

In the celebrated work \cite{Bri}, Bridgeland introduced and studied stability conditions on a triangulated category with the hope of rigorously defining $\mathcal{M}_{\mathrm{Kah}}(X)$.
He conjectured and confirmed in several important cases \cite{BB,Bri,Bri2, Bri3} that the string theorists' stringy K\"ahler moduli space $\mathcal{M}_{\mathrm{Kah}}(X)$ admits an embedding into the double quotient
$$
\mathrm{Aut}(\DD_X)\backslash\Stab(\DD_X)/\C
$$
of the space $\Stab(\DD_X)$ of Bridgeland stability conditions on the bounded derived category $\DD_X=\mathrm{D^bCoh}(X)$ of coherent sheaves on $X$. 

The purpose of this article is to provide a step toward differential geometric study of $\mathcal{M}_{\mathrm{Kah}}(X)$ via the space of Bridgeland stability conditions. 
Although we are mostly interested in geometric situations, we will use a more general categorical language in this article.   
On careful comparison of the two sides of Kontsevich's homological mirror symmetry $\mathrm{D^bCoh}(X) \cong \mathrm{D^bFuk}(Y)$, 
we will give a provisional definition of Weil--Petersson geometry and propose a conjecture which refine a previously known one.  
We will also provide some supporting evidence by computing basic geometric examples. 
A key example is the following, where our Weil--Petersson metric coincides with the classical Bergman metric on a Siegel modular variety (notations will be explained later).

\begin{Thm}[Theorem \ref{Stab-Siegel}, Corollary \ref{WP=Bergman}]
Let $A$ be the self-product $E_\tau \times E_\tau$ of an elliptic curve $E_\tau$. 
Then there is an identification 
$$
\overline{\Aut}_\CY(\DD_A)\backslash\Stab^+_\NN(\DD_A)/\C^\times\cong
\mathrm{Sp}(4,\Z) \backslash \mathfrak{H}_2.
$$
Moreover, the Weil--Petersson metric on the stringy K\"ahler moduli space $\overline{\Aut}_\CY(\DD_A)\backslash\Stab^+_\NN(\DD_A)/\C^\times$ is identified with
the Bergman metric on the  Siegel modular variety $\mathrm{Sp}(4,\Z) \backslash \mathfrak{H}_2$. 
\end{Thm}
This result is compatible with the mirror duality between $A$ and the principally polarized abelian surface. 
In fact, the complex moduli space of the latter is given by $\mathrm{Sp}(4,\Z) \backslash \mathfrak{H}_2$. 
Another example is a quintic threefold $X \subset \mathbb{P}^4$.  
Assuming that a conjectural Bridgeland stability condition exists, 
we will observe that the Weil--Petersson metric is given by a quantum deformation of the Poincar\'e metric near the large volume limit. 

It is worth noting that some aspects of the ideas in this article were presented in a series of Wilson's works \cite{TW,Wil,Wil2} on metrics on the complexified K\"ahler cones.  
In fact, his works are also motivated by mirror symmetry since the complexified K\"ahler cone is expected to give a local chart of $\mathcal{M}_{\mathrm{Kah}}(X)$ near a large volume limit. 
For example, in the case of a Calabi--Yau 3-fold $X$, the curvature of the so-called asymptotic Weil--Petersson metric 
was shown to be closely related to the trilinear form on $H^2(X,\Z)$ \cite{TW, Wil2}. 

On the other hand, an advantage of our approach taken in this article is the fact that our Weil--Petersson metric is global and makes perfect sense away from large volume limits, in contrast to Wilson's local study.  
As a matter of fact, the global aspects of the moduli space are of special importance in recent study of mirror symmetry and string theory. 

However, the real difficulty of the subject is to understand the precise relation between the stringy K\"ahler moduli space $\mathcal{M}_{\mathrm{Kah}}(X)$ and the space of Bridgeland stability conditions. 
Nevertheless, as an application of our work, we find a new condition on the conjectural embedding of $\mathcal{M}_{\mathrm{Kah}}(X)$ into a quotient of the space of Bridgeland stability conditions, 
namely the pullback of the Weil--Petersson metric on $\mathcal{M}_{\mathrm{Kah}}(X)$ should be non-degenerate (Conjecture \ref{Conj}). 
This simply means that the mirror identification $\mathcal{M}_{\mathrm{Kah}}(X) \cong \mathcal{M}_{\mathrm{cpx}}(Y)$ respects the Weil--Petersson geometries.  

\subsection*{Structure of article}
In Section \ref{Bridgeland}, we provide basic backgrounds on the Bridgeland stability conditions and twisted Mukai pairings.  
In Section \ref{WP}, after a brief review and reformulation of the classical Weil--Petersson geometry, we translate it in the context of Bridgeland stability conditions on a Calabi--Yau triangulated category. 
Section \ref{Ex} is the heart of the article. We carry out detailed calculations for two classes of abelian surfaces to justify our proposal. 
Lastly, we discuss the case of a quintic 3-fold, and comment on further research directions. 

\subsection*{Notation and conventions}
For an abelian group $A$, we denote by $A_\tf$ the quotient of $A$ by its torsion subgroup. 
We write $A_{K}:=A\otimes_\Z K$ for a field $K$. 
$\ch(-)$ denotes the Chern character and $\Td_X$ denotes the Todd class of $X$. 
We write $\Re(z)$ and $\Im(z)$ for the real and imaginary parts of $z$ respectively. 
A Calabi--Yau manifold is a complex manifold whose canonical bundle is trivial. 
Throughout the article, we work over complex numbers $\C$. 

\subsection*{Acknowledgement}
We are grateful to Shinobu Hosono, Yukinobu Toda and Jie Zhou for useful discussions. 
Special thanks go to Hiroshi Iritani for drawing our attention to the Gamma class.   
The present work was initiated when A. K. was supported by Harvard CMSA. 
It was supported in part by the Kyoto Hakubi Project and JSPS Grant-in-Aid Wakate(B)17K17817. 
Part of the work was done when Y.-W. F. was visiting Kavli IPMU and Kyoto University.
The work of S.-T. Y. was supported by the Simons Collaboration Grant on Homological Mirror Symmetry 385581, NSF grant DMS-1607871, and Harvard CMSA.


\section{Bridgeland stability conditions} \label{Bridgeland}

\subsection{Bridgeland stability conditions} \label{Stab}
The notion of stability conditions on a triangulated category $\DD$ was introduced by Bridgeland \cite{Bri}, 
following physical ideas of $\Pi$-stabilities of D-branes due to Douglas \cite{Dou}. 
Throughout this article, we assume that $\DD$ is essentially small, linear over the complex numbers $\C$, and of finite type, i.e. 
for every pair of objects $E$ and $F$, the $\C$-vector space $\oplus_{i}\Hom_{\DD}(E,F[i])$ is finite-dimensional.
Then we define the Euler form $\chi$ on the Grothendieck group $K(\DD)$ by the formula
$$
\chi(E,F):=\sum_{i}(-1)^i \dim_\C \Hom_{\DD}(E,F[i]). 
$$
The numerical Grothendieck group $\NN(\DD):=K(\DD)/K(\DD)^{\perp_{\chi}}$ is defined to be 
the quotient of $K(\DD)$ by the null space $K(\DD)^{\perp_{\chi}}$ of $\chi$. 
We also assume that $\DD$ is numerically finite, i.e. $\NN(\DD)$ is of finite rank. 
A large class of examples of such a triangulated category is provided by the bounded derived category of coherent sheaves $\DD_X=\mathrm{D^bCoh}(X)$ of a smooth projective variety $X$.  

\begin{Def}[\cite{Bri}]
A numerical stability condition $\sigma=(\ZZ,\PP)$ on a triangulated category $\DD$ consists of a group homomorphism $\ZZ:\NN(\DD)\rightarrow\C$ (central charge) 
and a collection of full additive subcategories $\PP=\{\PP(\phi)\}_{\phi \in \R}$ of $\DD$ (semistable objects) such that: 
\begin{enumerate}
\item If $0\neq E \in \PP(\phi)$, then $\ZZ(E)\in\R_{>0}\cdot e^{\sqrt{-1} \pi \phi}$.
\item $\PP(\phi+1)=\PP(\phi)[1]$. 
\item If $\phi_1>\phi_2$ and $A_i \in \PP(\phi_i)$, then $\Hom_{\DD}(A_1,A_2)=0$. 
\item For every $0 \ne E \in \mathcal{D}$, there exists a collection of exact triangles
$$
\xymatrix{
0=E_0 \ar[r] & E_1\ar[d]  \ar[r] & E_2 \ar[r] \ar[d]& \cdots \ar[r]& E_{k-1}\ar[r] & E \ar[d]\\
                    & A_1 \ar@{-->}[lu]& A_2 \ar@{-->}[lu] &  & & A_k  \ar@{-->}[lu] 
}
$$
such that $A_i \in \PP(\phi_i)$ and $\phi_1>\phi_2>\cdots>\phi_k$.
\item (Support property \cite{KS}) There exist a constant $C>0$ and a norm $||\ ||$
on $\NN(\DD)_\R$ such that $||E||\leq C|\ZZ(E)|$ for any semistable object $E$.
\end{enumerate}
\end{Def}


We denote by $\Stab_{\NN}(\DD)$ the space of numerical stability conditions on $\DD$. 
Bridgeland defined a nice topology on it such that the forgetful map
$$
\Stab_{\NN}(\DD)\longrightarrow \Hom(\NN(\DD),\C), \ \ \ \sigma=(\ZZ,\PP) \mapsto \ZZ
$$
is a local homeomorphism  \cite{Bri, KS}, i.e. deformations of the central charge lift uniquely to deformations of the stability condition. 
Thereby $\Stab_{\NN}(\DD)$ is naturally a complex manifold, which is locally modelled on the $\C$-vector space $\Hom(\NN(\DD),\C)$. 
Moreover, $\Stab_{\NN}(\DD)$ naturally carries a right action of the group $\widetilde{\GL^+(2,\R)}$, 
the universal cover of the group of orientation-preserving automorphism $\GL^+(2,\R)$ of the Euclidean plane $\R^2$,  
as well as a left action of the group $\Aut(\DD)$ of autoequivalences of $\DD$. 
The $\widetilde{\GL^+(2,\R)}$-action is given by post-composition on the central charge $\ZZ:\NN(\DD)\rightarrow \C \cong \R^2$ and a suitable relabelling of the phases. 
We often restrict this action to the subgroup $\C \subset\widetilde{ \GL^+(2,\R)}$ which acts freely.

\begin{Rem}\label{StabFuk}
Let $\mathrm{D^bFuk}(Y)$ be the derived Fukaya category of a Calabi--Yau manifold $Y$. 
We fix a holomorphic volume form $\Omega$ of $Y$. 
It is a folklore conjecture (c.f. \cite{Bri3}) that there exists a Bridgeland stability condition on
$\mathrm{D^bFuk}(Y)$ with central charge given by the period integral
$$
\ZZ(L)=\int_L\Omega. 
$$
Moreover, the special Lagrangian submanifolds of phase $\phi$ are the semistable objects of phase $\phi$ with respect to this stability condition. 
\end{Rem}

\subsection{Central charge via twisted Mukai pairing}\label{phyexp}
Let $X$ be a smooth projective variety. 
Motivated by work of Mukai in the case of K3 surfaces, C\u{a}ld\u{a}raru defined the the Mukai pairing on $H^*(X;\C)$ as follows \cite{Cal}: for $v,v'\in H^*(X;\C)$, 
$$
\langle v,v'\rangle_{\Muk}:=\int_X e^{c_1(X)/2}\cdot v^{\vee}\cdot v'.
$$
Here $v=\sum_j v_j\in\oplus_j H^j(X;\C)$ and its Mukai dual $v^{\vee}=\sum_j\sqrt{-1}^jv_j\in H^*(X;\C)$. 
Note that the above Mukai paring differs from Mukai's original one \cite{Muk} for K3 surfaces by a sign. 
We define a twisted Mukai vector of $E\in\DD_X=\mathrm{D^bCoh}(X)$ by 
$$
v_\Lambda(E):=\ch(E)\sqrt{\Td_X} \exp(\sqrt{-1}\Lambda)
$$
for any $\Lambda \in H^*(X;\C)$ such that $\Lambda^\vee=-\Lambda$. 
The usual Mukai vector is the special case where $\Lambda=0$. 
A twisted Mukai pairing is compatible with the Euler pairing; 
by the Hirzebruch--Riemann--Roch theorem, 
$$
\chi(E,F)=\int_X \ch(E^\vee) \ch(F)\Td_X= \langle v_\Lambda(E),v_\Lambda(F)\rangle_{\Muk}. 
$$
A geometric twisting $\Lambda$ compatible with the integral structure on the quantum cohomology was introduced by Iritani \cite{Iri} and Katzarkov--Kontsevich--Pantev \cite{KKP}. 
We shall use a reformulation due to Halverson--Jockers--Lapan--Morrison \cite{HJLM} in the following. 
First let us recall a familiar identity from complex analysis
$$
\frac{z}{1-e^{-z}}=e^{z/2}\frac{z/2}{\sinh(z/2)}=e^{z/2}\Gamma(1+\frac{z}{2 \pi \sqrt{-1}})\Gamma(1-\frac{z}{2\pi \sqrt{-1}}),  
$$
where $\Gamma(z)$ is the Gamma function. 
The power series in the LHS induces the Todd class $\Td_X$.
We then consider a square root of the Todd class by writing 
$$
\sqrt{\frac{z}{1-z}}\exp(\sqrt{-1}\Lambda(z))=e^{z/4}\Gamma(1+\frac{z}{2 \pi \sqrt{-1}}),
$$
and solve it for $\Lambda(z)$, where $z$ is real, as 
\begin{align}
\Lambda(z)&=\Im (\log\Gamma(1+\frac{z}{2 \pi \sqrt{-1}})) \notag \\
&=\frac{\gamma z}{2\pi}+\sum_{j\ge 1}(-1)^j \frac{\zeta(2j+1)}{2j+1}\left(\frac{z}{2\pi}\right)^{2j} \notag
\end{align}
where $\gamma$ is Euler's constant.
Since the constant term of $\Lambda(z)$ is zero, we may use it to define an additive characteristic class $\Lambda_X$, called the log Gamma class. 
Note that $\Lambda_X^\vee=-\Lambda_X$ as only odd powers of $z$ appear in $\Lambda(z)$. 
In the Calabi--Yau case, we can explicitly write it as
$$
\Lambda_X=-\frac{\zeta(3)}{(2 \pi)^3}c_3(X)+\frac{\zeta(5)}{(2\pi)^5}(c_5(X)-c_2(X)c_3(X))+\dots
$$
For K3 and abelian surfaces, there is no effect of twisting as $\Lambda_X=0$. 
For Calabi--Yau 3-folds, the modification is precisely given by the first term, which is familiar in period computations in the B-model side. 
We define $v_X(E)$ to be the twisted Muaki vector of an object $E \in \DD_X$ associated to the log Gamma class $\Lambda_X$, i.e. 
$$
v_X(E):=\ch(E)\sqrt{\Td_X} \exp(\sqrt{-1}\Lambda_X)
$$

Let $X$ be a projective Calabi--Yau manifold equipped with a complexified K\"ahler parameter 
$$
\omega=B+\sqrt{-1} \kappa \in H^2(X;\C),
$$   
where $\kappa$ is a K\"ahler class. Let also $q=\exp(2 \pi \sqrt{-1}\omega)$. 
We define $\exp_*( \omega)$ by 
$$
\exp_*( \omega):=1+ \omega+\frac{1}{2!}\omega*\omega+\frac{1}{3!}\omega*\omega*\omega+\cdots. 
$$
where $*$ denotes the quantum product. 
It is conjectured (c.f. \cite{JKLMR}) that near the large volume limit, which means that $\int_C\Im(\omega) \gg 0$ for all effective curve $C \subset X$,   
there exists a Bridgeland stability condition on $\DD_X$ with central charge of the form 
\begin{equation} \label{Central charge}
\ZZ(E)=- \left \langle \exp_*( \omega ), v_X(E) \right\rangle_{\Muk}. 
\end{equation}
Then the asymptotic behavior of the above central charge near the large volume limit is given by 
\begin{equation} \label{Asymp Z}
\ZZ(E) \sim -\int_X e^{- \omega} v_X(E)+ O(q). 
\end{equation} 
The existence of a Bridgeland stability condition with the asymptotic central charge given by the leading term of the above expression has been proved in various important examples 
including K3 surfaces and abelian surfaces \cite{Bri2}, as well as abelian threefolds \cite{BMS, MP}. 

\section{Weil--Petersson geometry} \label{WP}

\subsection{Classical Weil--Petersson geometry} 

We review some basics of the classical Weil--Petersson geometry on the complex moduli space $\mathcal{M}_{\mathrm{cpx}}(Y)$ of a projective Calabi--Yau $n$-fold $Y$. 
In fact, such a (possibly degenerate) metric can be defined on the complex moduli space of a polarized K\"ahler--Einstein manifold by the Kodaira--Spencer theory,  
but we will use a period theoretic method for a Calabi--Yau manifold (see for example \cite{Tia}). 
It fits naturally into the framework of the Hodge theory and gives us a connection to the Bridgeland stability conditions. 

First, we consider the following vector bundle on the complex moduli space $\mathcal{M}_{\mathrm{cpx}}(Y)$
$$
\mathcal{H}=\mathcal{R}^n\pi_*\underline{\C}\otimes\mathcal{O}_{\mathcal{M}_{\mathrm{cpx}}(Y)} \longrightarrow \mathcal{M}_{\mathrm{cpx}}(Y). 
$$
It comes equipped with a natural Hodge filtration $\{F^*\mathcal{H}\}$ of weight $n$. 
By the Calabi--Yau assumption, the first piece of the filtration defines a holomorphic line bundle $F^n\mathcal{H}\rightarrow\mathcal{M}_{\mathrm{cpx}}(Y)$, which is called the vacuum bundle. 
For a nowhere vanishing local section $\Omega$ of the vacuum bundle, the quantity 
$$
K^{\cpx}_{\WP}(z):=- \log \left((\sqrt{-1})^{n^2}\int_Y\Omega_z\wedge\overline{\Omega}_z\right)
$$
defines a local smooth function $K^{\cpx}_{\WP}$, known as the Weil--Petersson potential, on the complex moduli space $\mathcal{M}_{\mathrm{cpx}}(Y)$. 
Then the Weil--Petersson metric on $\mathcal{M}_{\mathrm{cpx}}(Y)$ is defined to be the Hessian metric $\frac{\sqrt{-1}}{2}\partial \bar{\partial} K^{\cpx}_{\WP}$. 
A fundamental fact is that the Hessian metric is non-degenerate and provides a canonical K\"ahler metric on the complex moduli space $\mathcal{M}_{\mathrm{cpx}}(Y)$. 

For Lagrangian submanifolds $L_1, L_2 \subset Y$, let $\chi(L_1,L_2)=\chi(HF^*(L_1,L_2))$ be the Euler pairing on the derived Fukaya category $\mathrm{D^bFuk}(Y)$.  
The following identity is useful for computing the Weil--Petersson potential.

\begin{Prop}\label{WPformula}
Provided that there exist formal sums of Lagrangian submanifolds $\{L_i\}$ representing a basis of $H_n(Y;\Z)_{\tf}$, 
we have 
\begin{equation}
K^{\cpx}_{\WP}(z)=- \log\Big((\sqrt{-1})^{-n} \sum_{i,j} \chi^{i,j}  \int_{L_i}\Omega_z \int_{L_j}\overline{\Omega_z}\Big),  \label{WP potential}  
\end{equation}
where $(\chi^{i,j})=(\chi(L_i,L_j))^{-1}$ is the inverse matrix. 
\end{Prop}
\begin{proof} 
Let $\{A_i\}$ be a basis of $H_n(Y;\Z)_{\tf}$.  
We define $(\gamma^{i,j})=(A_i\cdot A_j)^{-1}$ to be the inverse of the intersection matrix. 
Then by expanding $\Omega_z$ and $\overline{\Omega_z}$ by the dual basis of $\{A_i\}$, we can rewrite the Weil--Petersson potential as 
$$
K^{\cpx}_{\WP}(z)=- \log\Big((\sqrt{-1})^{n^2} \sum_{i,j} \gamma^{i,j}  \int_{A_i}\Omega_z \int_{A_j}\overline{\Omega_z}\Big).
$$
On the other hand, for Lagrangian submanifolds $L_1, L_2 \subset Y$, the identity
$$
[L_1]\cdot[L_2]=(\sqrt{-1})^{n(n+1)}\chi(L_1,L_2),
$$
is standard in the Lagrangian Floer theory (see for example \cite[Section 4.3]{GPS}). 
This completes the proof. 
\end{proof}

An advantage of the new expression (Equation (\ref{WP potential})) is that it is not only categorical but also Hodge theoretic in the sense that it is written in terms of period integrals. 

\subsection{Weil--Petersson geometry on $\Stab_{\NN}(\DD)$}\label{WPStab}

Motivated by Remark \ref{StabFuk} and Proposition \ref{WPformula}, we shall propose a definition of Weil--Petersson geometry 
on a suitable quotient of the space of Bridgeland stability conditions on a Calabi--Yau triangulated category $\DD$ of dimension $n \in \N$, i.e. 
for every pair of objects $E$ and $F$, there is a natural isomorphism
$$
\mathrm{Hom}^{*}_{\DD}(E,F)\cong\mathrm{Hom}^*_{\DD}(F,E[n])^{\vee}.
$$
An important consequence is that the Euler form on $\NN(\DD)$ is (skew-)symmetric if $n$ is even (odd).

Let $\{E_i\}$ be a basis of the numerical Grothendieck group $\NN(\DD)$.  
We define a bilinear form $\mathfrak{b}: \Hom(\NN(\DD),\C)^{\otimes 2} \rightarrow \C$ by
$$
\ZZ_1 \otimes \ZZ_2 \mapsto \mathfrak{b}(\ZZ_1,\ZZ_2):=\sum_{i,j} \chi^{i,j}   \ZZ_1(E_i)\ZZ_2(E_j),  
$$
where $(\chi^{i,j}):=(\chi(E_i,E_j))^{-1}$.  
It is an easy exercise to check that the bilinear form $\mathfrak{b}$ is independent of the choice of a basis. 

\begin{Def} \label{Stab+}
We define the subset $\Stab_{\NN}^+(\DD) \subset \Stab_{\NN}(\DD)$ by 
$$
\Stab_{\NN}^+(\DD):=\{ \sigma=(\ZZ,\PP) \ | \ \mathfrak{b}(\ZZ,\ZZ)=0, \ (\sqrt{-1})^{-n} \mathfrak{b}(\ZZ,\overline{\ZZ})>0\}.  
$$
\end{Def}
The first condition is vacuous when $n$ is odd as the bilinear form $\mathfrak{b}$ is skew-symmetric. 
Such conditions have been studied in the case of K3 surfaces (a dual description via the Mukai pairing) under the name of {\it reduced} stability conditions \cite{Bri2}. 
We note that $\Stab_{\NN}^+(\DD)$ is an analogue of a period domain in the Hodge theory, and the natural free $\C$-action on $\Stab_{\NN}(\DD)$ preserves the subset $\Stab_{\NN}^+(\DD)$.

\begin{Def}\label{WPpotential}
Let $s=(\ZZ_{\bar{\sigma}},\PP_{\bar{\sigma}})$ be a local holomorphic section of the $\C$-torsor $\Stab_{\NN}^+(\DD)\rightarrow\Stab_{\NN}^+(\DD)/\C$, then
$$
K_{\WP}(\bar{\sigma}):= -\log \Big((\sqrt{-1})^{-n} \mathfrak{b}(\ZZ_{\bar{\sigma}},\overline{\ZZ_{\bar{\sigma}}}) \Big)
$$ 
defines a local smooth function on $\Stab_{\NN}^+(\DD)/\C$. 
We call it the Weil--Petersson potential on $\Stab_{\NN}^+(\DD)/\C$. 
\end{Def}

\begin{Prop}
The complex Hessian $\frac{\sqrt{-1}}{2}\partial\overline{\partial}K_{\WP}$ of the Weil--Petersson potential $K_{\WP}$ does not depend on the choice of a local section $s$. 
Moreover, it descends to the double quotient space 
$$
\Aut(\DD) \backslash\Stab_{\NN}^+(\DD)/ \C
$$
away from singular loci. 
\end{Prop}
\begin{proof}
The first assertion is standard. 
The second assertion follows from the fact that autoequivalences are compatible with the Euler pairing,
and the induced actions on the numerical Grothendieck group $\NN(\DD)$ send a basis to another basis.
Therefore the local sections which are identified by elements of $\Aut(\DD)$ differ only by multiplying local holomorphic functions, and thus the well-definedness of the metric follows
from that of $\mathfrak{b}$.
\end{proof}

The situation is particularly interesting when $n$ is odd, as $\Stab_{\NN}(\DD)$ naturally carries a holomorphic symplectic structure.  
Given a symplectic basis $\{E_i,F_i\}$ of $\NN(\DD)$, 
the skew-symmetric bilinear form $\mathfrak{b}: \Hom(\NN(\DD),\C)^{\otimes 2} \rightarrow \C$ is simply
$$
\ZZ_1 \otimes \ZZ_2 \mapsto \sum_i \Big( \ZZ_1(F_i)\ZZ_2(E_i) - \ZZ_1(E_i)\ZZ_2(F_i) \Big),
$$
which provides a nowhere vanishing holomorphic $2$-form on $\Stab_{\NN}(\DD)$.  

\begin{Ex} \label{ell curve}
As a sanity check, we shall carry out the above construction for the derived category $\DD_X=\mathrm{D^bCoh}(X)$ of an elliptic curve $X$. 
Since the action of $\widetilde{\GL^+(2,\R)}$ on $\Stab_{\NN}(\DD_X)$ is free and transitive \cite[Theorem 9.1]{Bri}, we have 
$$
\Stab_{\NN}^+(\DD_X)=\Stab_{\NN}(\DD_X)  \cong \widetilde{\GL^+(2,\R)} \cong \C \times \HH,
$$
as a complex manifold. Thus the double quotient  is
$$
\Aut(\DD_X) \backslash\Stab_{\NN}^+(\DD_X)/ \C \cong \mathrm{PSL}(2,\Z) \backslash \HH. 
$$
This is indeed the K\"ahler moduli space of the elliptic curve $X$.
The normalized central charge at $\tau\in\HH$ is given by
$$
\ZZ(E)=-\deg(E)+ \tau \cdot \rank(E).
$$ 
Hence the Weil--Petersson potential is 
\begin{align}
K_{\WP}(\tau)&=-\log\Big((\sqrt{-1})^{-1}
(\mathcal{Z}(\mathcal{O}_p)\overline{\mathcal{Z}}(\mathcal{O}_E)-
\mathcal{Z}(\mathcal{O}_E)\overline{\mathcal{Z}}(\mathcal{O}_p))\Big) \notag \\
&=-\log(\Im(\tau))-\log2. \notag
\end{align}
This is the Poincar\'e potential on $\HH$ and it descends to the K\"ahler moduli space $\mathrm{PSL}(2,\Z) \backslash \HH$.   
\end{Ex}

\subsection{Refining conjecture} 
Let $X$ be a projective Calabi--Yau $n$-fold.  
Then $\Stab_{\NN}(\DD_X)$ can be considered as an extended version of the stringy K\"ahler moduli space $\mathcal{M}_{\mathrm{Kah}}(X)$ \cite[Section 7.1]{Bri3}.
It is akin to the big quantum cohomology rather than the small quantum cohomology 
in the sense that the tangent space of $\mathcal{M}_{\mathrm{Kah}}(X)$ is $H^{1,1}(X)$ while that of $\Stab_{\NN}(\DD_X)$ is $\oplus_p H^{p,p}(X)$. 
It is conjectured by Bridgeland \cite{Bri} that there should exist an embedding of $\mathcal{M}_{\mathrm{Kah}}(X)$ into 
$$
\mathrm{Aut}(\DD_X) \backslash \Stab_{\NN}(\DD_X)/ \C. 
$$

Note that when $n$ is odd, the double quotient is a holomorphic contact space thanks to the holomorphic symplectic structure on $\Stab_{\NN}(\DD_X)$.

Motivated by mirror symmetry and classical Weil--Petersson geometry, especially the fact that Weil--Petersson metric is non-degenerate on $\mathcal{M}_{\mathrm{cpx}}(X)$,
we can now propose the following, which refines the previous conjecture.

\begin{Conj} \label{Conj} 
There exists an embedding of the stringy K\"ahler moduli space
$$
\iota:\mathcal{M}_{\mathrm{Kah}}(X)\hookrightarrow\mathrm{Aut}(\DD_X) \backslash\Stab_{\NN}^+(\DD_X)/ \C. 
$$
The complex Hessian of the pullback $\iota^*K_{\WP}$ of the Weil--Petersson potential $K_{\WP}$ defines a K\"ahler metric on $\mathcal{M}_{\mathrm{Kah}}(X)$, 
i.e. non-degenerate. 
Moreover, it is identified with the Weil--Petersson metric on the complex moduli space $\mathcal{M}_{\mathrm{cpx}}(Y)$ of a mirror manifold $Y$ 
under the mirror map $\mathcal{M}_{\mathrm{Kah}}(X) \cong \mathcal{M}_{\mathrm{cpx}}(Y)$. 
When $n=3$, the image of $\mathcal{M}_{\mathrm{Kah}}(X)$ is locally a holomorphic Legendre variety. 
\end{Conj}

We checked that the conjecture holds for the elliptic curves (Example \ref{ell curve}) and will provide more supporting evidence in the next section. 
It is worth noting that the real difficulty lies in providing a mathematical definition of the stringy K\"ahler moduli space $\mathcal{M}_{\mathrm{Kah}}(X)$. 
One potential application of the above conjecture is that,
we can make use of the non-degeneracy condition on the Weil--Petersson metric to characterize $\mathcal{M}_{\mathrm{Kah}}(X)$.


\section{Computation} \label{Ex}

We begin our discussion with Calabi--Yau surfaces, for which there is a mathematical definition of stringy K\"ahler moduli space via the Bridgeland stability conditions \cite[Section 7]{BB}. 
Our computation heavily relies on existing deep results, mainly due to Bridgeland, and we do not claim originality. 
The purpose of this section is to back up our conjecture by concrete examples. 

\subsection{Self-product of elliptic curve}
We consider the self-product $A:=E_{\tau}\times E_{\tau}$ of an elliptic curve $E_\tau:=\C/(\Z +\tau \Z)$ for a generic $\tau \in \HH$. 
We denote by $\NS(A):=H^2(A,\Z) \cap H^{1,1}(A)$ the N\'eron--Severi lattice of $A$.  
Before considering the space of stability conditions, let us take a look at the set of complexified K\"ahler forms $\omega \in \NS(A)_\C$. 
Let $dz_1$ be a basis of $H^{1,0}(E_\tau)$ of the first $E_\tau$ factor of $A$, and $dz_2$ similarly. 
Then a complexified K\"ahler form $\omega$ can be expressed as 
$$
\omega = \sqrt{-1} \left(\rho dz_1\wedge d\bar{z}_1 + \tau dz_2 \wedge d\overline{z}_2 + \sigma (dz_1 \wedge d\bar{z}_2 - d\bar{z}_1 \wedge dz_2)\right)
$$
such that the imaginary part $\Im(\omega)$ is a K\"ahler form. 
The real part $\Re(\omega)$ is often called a B-field. 
Let $\mathfrak{H}_g$ be the Siegel upper half-space of degree $g$ defined by 
$$
\mathfrak{H}_g:=\{M \in \mathrm{M}(g,\C) \ | \ M^t=M, \Im(M)>0 \}.
$$
In this abelian surface example, we do not fix a polarization, but we vary it in a 3-dimensional space $\mathfrak{H}_2$ as follows.

\begin{Prop}[\cite{KL}] \label{identify}
Let $A_g:=E_\tau^{\times g}$ be the self-product of $g$ copies of an elliptic curve $E_\tau$
The set of complexified K\"ahler forms can be identified with the Siegel upper-half space $\mathfrak{H}_g$ of genus $g$. 
In the $g=2$ case, the identification is given by the assignment 
$\omega \mapsto M_\omega:=\begin{bmatrix}
                \rho   &  \sigma  \\
                 \sigma& \tau  \\
\end{bmatrix}$. 
\end{Prop}

\begin{proof}
The $g=1$ case is standard. Suppose that $g=2$. 
It suffices to show that $\Im(M_\omega)>0$. 
Since $\omega$ is a complexified K\"ahler form, we have 
$$
\tr(\Im(M_\omega))=\Im(\rho)+\Im(\tau)=\int_{E_\tau \times \mathrm{pt}}\Im(\omega)+\int_{\mathrm{pt} \times E_\tau}\Im(\omega) >0
$$
and 
$$
\det \left(\Im(M_\omega)\right)=\Im(\rho)\Im(\tau)-\Im(\sigma)^2=\Im(\omega)^2>0, 
$$ 
and thus $M_\omega \in \mathfrak{H}_2$. 

On the other hand, since $A_2$ contains no rational curves, the K\"ahler cone coincides with the connected component of the positive cone in $\NS(S)$ which contains a K\"ahler class. 
This readily proves the assertion for $g=2$.  
Since we do not need the higher genus case, we leave a proof to the reader (c.f. \cite[Section 6]{KL}).  
\end{proof}

We now recall some notations in \cite{Bri2}.
A result of Orlov \cite[Proposition 3.5]{Orl2} shows that every autoequivalence of $\DD_A$ induces a Hodge isometry of the lattice $H^*(A;\Z)$ equipped with the Mukai pairing.
Hence there is a group homomorphism
$$
\delta:\Aut(\DD_A)\longrightarrow\Aut H^*(A;\Z).
$$
The kernel of the homomorphism will be denoted by $\Aut^0(\DD_A)$.

Let $\Omega\in H^2(A;\C)$ be the class of a nonzero holomorphic two-form on $A$.
The sublattice
$$
\NN(A):= H^*(A;\Z) \cap \Omega^{\perp} \subset H^*(A;\C)
$$
can be identified with $\NN(\DD_A)=H^0(A;\Z)\oplus \NS(A) \oplus H^4(A;\Z)$ and has signature $(3,2)$.
In fact, since the complex moduli $\tau \in \HH$ is generic, $\NN(\DD_A)\cong U^{\oplus 2}\oplus \langle 2 \rangle $ as a lattice. 
Here $U$ is the hyperbolic lattice, and $\langle 2\rangle$ denotes an integral lattice of rank 1 with the Gram matrix $(2)$.

We define a subset $\PP(A) \subset \NN(\DD_A)_\C$ consisting of vectors $\mho \in \NN(\DD_A)_\C$ 
whose real and imaginary parts span a negative definite $2$-plane in $\NN(\DD_A)_\R$.  
This subset has two connected components.
We denote by $\PP^+(A)$ the component containing vectors of the form 
$\mho_\omega:=\exp(\omega)$ 
for a complexified K\"ahler class $\omega \in \NS(A)_\C$.

Let us review some results on the space of Bridgeland stability conditions on algebraic surfaces following \cite{Bri2}.
The central charge of a numerical stability condition is of the form
$$
\ZZ_\mho(E)=-\langle \mho, v_A(E) \rangle_{\Muk}=-\langle \mho,  \ch(E) \rangle_{\Muk}
$$
for some $\mho \in \NN(\DD_A)_\C$. 
When $\mho=\mho_{\omega}$ for some complexified K\"ahler class $\omega$, 
one can construct a stability condition with central charge $\ZZ_{\mho_{\omega}}$ using the tilting theory and Bogomolov inequality.
Moreover, such stability conditions are \emph{geometric} in the sense that all skyscraper sheaves are stable and of the same phase.
We denote by $\Stab^\dagger(A) \subset \Stab(A)$  the connected component containing the set of geometric stability conditions.
The following result on the global structure of $\Stab^\dagger(A)$ is due to Bridgeland \cite[Section 15]{Bri2}.

\begin{Thm}[\cite{Bri2}] \label{BriK3}
Let $A$ be an abelian surface over $\C$.
\begin{enumerate}
\item The forgetful map $\pi$ sending a stability condition to the associated vector $\mho \in \NN(\DD_A)_\C$ maps onto the open subset $ \PP^+(A)\subset\NN(\DD_A)_\C$.
Moreover, the map
$$
\pi:\Stab^\dagger(A) \longrightarrow \PP^+(A).  
$$
is the universal cover of $\PP^+(A)$ with the group of deck transformations generated by the double shift functor $[2] \in \Aut(\DD_A)$.

\item The action of $\Aut(\DD_A)$ on $\Stab(A)$ preserves the connected component $\Stab^\dagger(A)$.

\item The group $\Aut^0(\DD_A)$ is generated by the double shift functor $[2]$, together with twists by elements of $\mathrm{Pic}^0(A)$, and pullbacks by automorphisms of $A$ acting trivially on $H^*(A;\Z)$.
Note that twists by elements of $\mathrm{Pic}^0(A)$ and pullbacks by automorphisms of $A$ acting trivially on $H^*(A;\Z)$, act trivially on $\Stab^\dagger(A)$.

\item There exists a short exact sequence of groups
$$
1\longrightarrow\Aut^0(\DD_A)\longrightarrow\Aut(\DD_A)\longrightarrow\Aut^+H^*(A;\Z)\longrightarrow1,
$$
where $\Aut^+H^*(A;\Z)\subset\Aut H^*(A;\Z)$ is the index 2 subgroup consisting of elements which do not exchange the two components of $\PP(A)$.

\end{enumerate}
\end{Thm}

Following our proposal in the previous section, it is natural to consider the following subset of $\Stab^\dagger(A)$.
$$
\Stab_{\NN}^+(A):=\{ (\ZZ,\PP)\in\Stab^\dagger(A) \ | \ \mathfrak{b}(\ZZ,\ZZ)=0, \  -\mathfrak{b}(\ZZ,\overline{\ZZ})>0\}
$$
(c.f. Definition \ref{Stab+}).
Hence we need to compute the bilinear form $\mathfrak{b}$.
We start with a lemma.

\begin{Lem}\label{LemMuk}
Let $X$ be a smooth projective variety of dimension $n$. 
Recall the twisted Mukai vector $v_X(E)=\ch(E)\sqrt{\Td_X}\exp(\sqrt{-1}\Lambda_X)$. 
Then
\begin{enumerate}
\item Assume that $X$ is Calabi--Yau, then $\langle v, w\rangle_{\Muk}=(-1)^n\langle w,v\rangle_{\Muk}$ for any $v,w\in H^{2*}(X;\C)$.
\item 
Let $\{E_i\}$ be a basis of the numerical Grothendieck group $\NN(D_X)$. Then
$$
\sum_{i,j} \langle v, v_X(E_i)\rangle _{\Muk}\cdot\chi^{i,j}\cdot\langle v_X(E_j),w\rangle_{\Muk}=\langle v,w\rangle_{\Muk},
$$
where $(\chi^{i,j})=(\chi(E_i,E_j))^{-1}$. 
	\end{enumerate}
\end{Lem}

\begin{proof}
The first assertion follows directly from the Serre duality.
The second assertion is a simple linear-algebraic fact which follows from the identity $\chi(E_i,E_j)=\langle v_X(E_i),v_X(E_j)\rangle_{\Muk}$. 
\end{proof}

Note that Lemma \ref{LemMuk} is purely algebraic and thus holds in a categorical setting as well.

We now compute the bilinear form $\mathfrak{b}$, with central charge of the form 
$$
\ZZ_{\mho}(E)=-\langle\mho, v_X(E)\rangle_{\Muk},
$$
where $\mho\in H^*(X;\C)$.

\begin{Lem}\label{b-mho}
$
\mathfrak{b}(\ZZ_{\mho_1},\ZZ_{\mho_2})=(-1)^n\langle\mho_1,\mho_2\rangle_{\Muk}.
$
\end{Lem}

\begin{proof}
This is a simple application of Lemma \ref{LemMuk}. We have
\begin{align}
\mathfrak{b}(\ZZ_{\mho_1},\ZZ_{\mho_2}) 
&= \sum_{i,j} \chi^{i,j} \cdot \langle \mho_1, v_X(E_i)\rangle_{\Muk} \cdot \langle \mho_2, v_X(E_j)\rangle_{\Muk} \notag \\
&= (-1)^{n}\sum_{i,j} \chi^{i,j} \cdot \langle \mho_1, v_X(E_i)\rangle_{\Muk}\cdot \langle v_X(E_j), \mho_2\rangle_{\Muk}  \notag \\
&=(-1)^n\langle\mho_1,\mho_2\rangle_{\Muk}. \notag
\end{align}
\end{proof}

Using Lemma \ref{b-mho}, we can determine which stability conditions lie in the subset $\Stab_{\NN}^+(A) \subset \Stab_{\NN}(A)$.

\begin{Prop}\label{WPdomain}
A stability condition $(\ZZ_{\mho},\PP)\in\Stab^\dagger(A)$ lies in $\Stab_{\NN}^+(A)$ if and only if $\langle\mho,\mho\rangle=0$ and $-\langle\mho,\overline{\mho}\rangle>0$.
This condition is equivalent to that $\mho$ is of the form $\mho_{\omega}=c\exp(\omega)$ for some constant $c\in\C$ and complexified K\"ahler class $\omega$.
\end{Prop}

\begin{proof}
The first assertion follows directly from Lemma \ref{b-mho}.
The second assertion follows from an explicit calculation of the Mukai pairing on surfaces.
\end{proof}

Note that central charge of the form
$$
\ZZ_{\mho_{\omega}}(E)=-\langle\mho_{\omega}, v_A(E)\rangle_{\Muk}
$$
where $\mho_{\omega}=\exp(\omega)$ for a complexified K\"ahler class $\omega$, has been discussed in physics literatures, see Section \ref{phyexp}.
One can prove that a stability condition with such a central charge always lies in $\Stab^+_{\NN}(\DD)$.

\begin{Prop}\label{WP_B}
\begin{enumerate}
\item The central charge $\ZZ_{\mho_\omega}$ satisfies
$$
\mathfrak{b}(\ZZ_{\mho_\omega},\ZZ_{\mho_\omega})=0, \ \ (\sqrt{-1})^{-n}\mathfrak{b}(\ZZ_{\mho_\omega},\overline{\ZZ_{\mho_\omega}})>0.
$$ 
\item 
The Weil--Petersson potential at a stability condition with central charge $\ZZ_{\mho_\omega}$ is given by
$$
K_{\WP}(\omega)=-\log(\Im(\omega)^n)-\log\Big(\frac{2^n}{n!}\Big).
$$
\end{enumerate}
\end{Prop}

\begin{proof}
\begin{enumerate}
\item By applying Lemma \ref{b-mho}, we have
\begin{align}
(\sqrt{-1})^{-n}\mathfrak{b}(\ZZ_{\mho_\omega},\overline{\ZZ_{\mho_\omega}}) 
&=(\sqrt{-1})^{n}\langle \mho_\omega, \mho_{\overline{\omega}}\rangle_{\Muk} \notag\\
&=  \frac{(\sqrt{-1})^n}{n!}(-\omega+\overline{\omega})^n  \notag \\
&= \frac{2^n}{n!}\Im(\omega)^n>0  \notag.
\end{align}
The last inequality is a consequence of the fact that $\omega$ is a complexified K\"ahler class. 
The equality $\mathfrak{b}(\ZZ_{\mho_\omega},\ZZ_{\mho_\omega})=0$ also follows from the above by replacing $\overline{\ZZ_{\mho_\omega}}$ by $\ZZ_{\mho_\omega}$. 
\item The assertion follows from the above explicit computation. 
\end{enumerate}
\end{proof}

Motivated by Proposition \ref{WPdomain}, we define the following subset $\RR^+(A)$ of $\PP^+(A)$.
$$
\RR^+(A):=\{\mho\in\PP^+(A)\ |\ \ \langle \mho, \mho \rangle_{\Muk}=0, \ -\langle \mho, \overline{\mho} \rangle_{\Muk}>0 \}.
$$
By Theorem \ref{BriK3} (1), the forgetful map $\Stab^+_\NN(\DD_A)\rightarrow\RR^+(A)$ is a covering map 
with the group of deck transformation generated by the double shift functor $[2] \in \Aut(\DD_A)$.

\begin{Lem} \label{Tube domain}
$\RR^+(A)/\C^\times \cong \mathfrak{H}_2$ as a complex manifold. 
Thereby we have an identification
$$
\langle[2]\rangle\backslash\Stab^+_\NN(\DD_A)/\C^\times\cong\mathfrak{H}_2.
$$
\end{Lem}

\begin{proof}
The quotient
$$
\RR^+(A)/\C^\times \cong \{ \C \mho \in \mathbb{P}(\NN(\DD_A)_\C) \ | \ \langle \mho, \mho \rangle_{\Muk}=0, \ -\langle \mho, \overline{\mho} \rangle_{\Muk}>0 \}
$$
is the symmetric domain of type $\mathrm{IV}_3$.
The assertion follows from the standard identification $\mathrm{IV}_3 \cong \mathrm{III}_{2}$ of the symmetric domains. 
More explicitly, it is given by the tube domain realization $\mathfrak{H}_2 \rightarrow \RR^+(A)/\C^\times: \omega \mapsto [\mho_\omega]$. 
\end{proof}

We now recall the definition of \emph{Calabi--Yau autoequivalences} following the work of Bayer and Bridgeland \cite{BB}.
Define
$$
\Aut_{\CY}^+H^*(A)\subset\Aut^+H^*(A)
$$
to be the subgroup of Hodge isometries which preserve the class of holomorphic 2-form $[\Omega]\in\mathbb{P}H^*(A;\C)$.
Any such isometry restricts to give an isometry of $\NN(\DD_A)$.
In fact,
$$
\Aut_{\CY}^+H^*(A)\subset\Aut\NN(\DD_A)
$$
is the subgroup of index two which do not exchange the two components of $\PP(A)$.

An autoequivalence $\Phi\in\Aut(\DD_A)$ is said to be \emph{Calabi--Yau} if the induced Hodge isometry $\delta(\Phi)$ lies in $\Aut_{\CY}^+H^*(A)$.
We denote $\Aut_\CY(\DD_A)\subset\Aut(\DD_A)$ the group of Calabi--Yau autoequivalences.
By Theorem \ref{BriK3} (4), there exists a short exact sequence
$$
1\longrightarrow\Aut^0(\DD_A)\longrightarrow\Aut_\CY(\DD_A)\longrightarrow
\Aut_\CY^+H^*(A)\longrightarrow1.
$$

We write $\Aut_{\mathrm{tri}}^0(\DD_A)\subset\Aut^0(\DD_A)$ for the subgroup generated by twists by elements of $\mathrm{Pic}^0(A)$ and pullbacks by automorphisms of $A$ acting trivially on $H^*(A;\Z)$.
Recall from Theorem \ref{BriK3} (3) that $\Aut_{\mathrm{tri}}^0(\DD_A)$ acts trivially on $\Stab^\dagger(\DD_A)$.
We define
$$
\overline{\Aut}_\CY(\DD_A):=\Aut_\CY(\DD_A)/\Aut_{\mathrm{tri}}^0(\DD_A).
$$
Then $\overline{\Aut}_\CY(\DD_A)$ acts on $\Stab^+_\NN(\DD_A)$,  and there is a short exact sequence
$$
1\longrightarrow\langle[2]\rangle\longrightarrow\overline{\Aut}_\CY(\DD_A)\longrightarrow
\Aut_\CY^+H^*(A)\longrightarrow1.
$$

\begin{Thm}\label{Stab-Siegel}
The covering map $\pi$ induces an isomorphism
$$
\overline{\Aut}_\CY(\DD_A)\backslash\Stab^+_\NN(\DD_A)/\C^\times\cong
\mathrm{Sp}(4,\Z) \backslash \mathfrak{H}_2
$$
between the double quotient of $\Stab^\dagger_\NN(\DD_A)$ and the Siegel modular variety $\mathrm{Sp}(4,\Z) \backslash \mathfrak{H}_2$. 
We will call it the stringy K\"ahler moduli space of $A$.
\end{Thm}

\begin{proof}
From the previous discussions, we have
$$
\overline{\Aut}_\CY(\DD_A)\backslash\Stab^+_\NN(\DD_A)/\C^\times\cong
\Aut_\CY^+H^*(A)\backslash \mathfrak{H}_2.
$$
The action of $\Aut_\CY^+H^*(A)$ on $\mathfrak{H}_2$ is purely lattice theoretic. 
As an abstract group $\Aut_\CY^+H^*(A) \cong O^+(U^{\oplus 2} \oplus \langle 2 \rangle)$. 
By a fundamental result \cite[Lemma 1.1]{GN} of Gritsenko and Nikulin, it can be identified with the standard $\mathrm{Sp}(4,\Z)$-action on the Siegel upper-half space $\mathfrak{H}_2$.
\end{proof}

\begin{Rem}
It is shown in \cite{KL} that $A_g$ is mirror symmetric to a principally polarized abelian surface of dimension $g$. 
Theorem \ref{Stab-Siegel} is thereby compatible with the fact 
that the Siegel modular variety $\mathrm{Sp}(2g,\Z) \backslash \mathfrak{H}_g$ is precisely the complex moduli space of principally polarized abelian surfaces of dimension $g$.  
For $g>2$, we expect that $\mathrm{Sp}(2g,\Z) \backslash \mathfrak{H}_g$ is covered 
by a similar double quotient of a suitable subset of $\Stab(\DD_{A_g})$.
\end{Rem}

There exists a canonical metric on the Siegel modular variety $\mathrm{Sp}(4,\Z) \backslash \mathfrak{H}_2$, namely the Bergman metric. 
It is known to be a complete K\"ahler--Einstein metric. 
The main theorem of this section is to show that the Bergman metric coincides with the Weil--Petersson metric defined by Definition \ref{WPpotential}.

\begin{Prop}[\cite{Sch}]
The Bergman kernel $K_{\Ber}:\mathfrak{H}_g \times \mathfrak{H}_g \rightarrow \C$ of the Siegel upper half-space $\mathfrak{H}_g$ of degree $g$ is given by
$$
K_{\Ber}(M,N)=-\tr(\log(-\sqrt{-1}(M-\overline{N}))). 
$$
The Bergman metric is defined to be the complex Hessian of the Bergman potential
$$
K_{\Ber}(M):=K_{\Ber}(M,M)=-\tr(\log(2\Im(M))). 
$$
\end{Prop}

\begin{Thm}
The Weil--Petersson potential on $\Stab^+_\NN(\DD_A)/\C$ coincides with the Bergman potential of the Siegel upper half-plane $\mathfrak{H}_2$ up to a constant.
\end{Thm}

\begin{proof}
By Proposition \ref{WPdomain}, the central charge of a stability condition in $\Stab^+_\NN(\DD_A)$ is of the form $\ZZ(E)=-c\langle\mho_\omega,v_A(E)\rangle$ for some complexified K\"ahler class $\omega$.
So we can apply the calculation of the Weil--Petersson potential in Proposition \ref{WP_B}.

The key idea is to use the identification of $\omega$ and $M_\omega$ provided in Proposition \ref{identify}. 
Then the two K\"ahler potentials are related as follows: 
\begin{align}
K_{\WP}(\omega) &= - \log(\Im(\omega)^2) - \log 2 \notag \\
& = - \log(\det (\Im(M_\omega)))-\log 2 \notag \\
&=- \tr(\log(2 \Im(M_\omega)))+\log 2 \notag \\
&= K_{\Ber}(M_\omega)+\log 2. \notag  
\end{align}
This completes the proof. 
\end{proof}
 
\begin{Cor} \label{WP=Bergman}
The Weil--Petersson metric on the stringy K\"ahler moduli space is identified with the Bergman metric on $\mathrm{Sp}(4,\Z) \backslash \mathfrak{H}_2$ via the isomorphism in Theorem \ref{Stab-Siegel}.
\end{Cor}

\subsection{Split abelian surfaces}
Now let  $A$ be a split abelian surface,
that is, $A\cong E_{\tau_1}\times E_{\tau_2}$ for elliptic curves $E_{\tau_1}$ and $E_{\tau_2}$. 
Such a splitting is unique provided that $E_{\tau_1}$ and $E_{\tau_2}$ are generic, or equivalently $\NS(A)\cong U$. 

Discussions in the previous section carries over for the split abelian surface $A$. 
The set of complexified K\"ahler forms is identified with $\HH \times \HH$, which is diagonally embedded in $\mathfrak{H}_2$ (c.f. Proposition \ref{identify}). 
It is precisely the symmetric domain of type $\mathrm{IV}_2$ associated to $U^{\oplus 2}$. 
On the other hand, it is known (c.f. \cite[Proposition 2.6]{HOLY}) that 
$$
{\Aut}^+_\CY H^*(A)  \cong O^+(U^{\oplus 2})  \cong \mathrm{P}(\mathrm{SL}(2,\Z)\times \mathrm{SL}(2,\Z)) \rtimes \Z_2, 
$$
where $\mathrm{P}(\mathrm{SL}(2,\Z)\times \mathrm{SL}(2,\Z))$ represents the quotient group of $\mathrm{SL}(2,\Z)\times \mathrm{SL}(2,\Z)$ by the involution $(A,B) \mapsto (-A,-B)$ 
and the semi-direct product structure is given by the generator of $\Z_2$ acting on $\mathrm{SL}(2,\Z)\times \mathrm{SL}(2,\Z)$ by exchanging the two factors. 

\begin{Thm}
There is an identification 
$$
\overline{\Aut}_\CY(\DD_A)\backslash\Stab^+_\NN(\DD_A)/\C^\times\cong
\mathrm{P}(\mathrm{SL}(2,\Z)\times \mathrm{SL}(2,\Z)) \rtimes \Z_2 \backslash (\HH \times \HH)
$$
Moreover, the Weil--Petersson metric on the stringy K\"ahler moduli space $\overline{\Aut}_\CY(\DD_A)\backslash\Stab^+_\NN(\DD_A)/\C^\times$ is identified with
the Bergman metric on the Siegel modular variety $\mathrm{P}(\mathrm{SL}(2,\Z)\times \mathrm{SL}(2,\Z)) \rtimes \Z_2 \backslash (\HH \times \HH)$. 
\end{Thm}

This observation is compatible with self-mirror symmetry for the split abelian surfaces. 
In fact a lattice polarized version of the global Torelli Theorem asserts that the complex moduli space of split abelian surfaces are given by the above Siegel modular variety.

\begin{Rem}
A similar computation can be carried out for $M$-polarized K3 surfaces for the lattice $M \cong U^{\oplus2} \oplus \langle -2 \rangle$ or $U^{\oplus2}$.  
The main difference is that there are spherical objects in the derived category $\DD_X$ of a K3 surface $X$ and we need to remove the union of certain hyperplanes from $\PP^+$. 
Moreover, the subgroup of $\Aut^0(\DD_X)$ which preserves the connected component $\Stab^\dagger(\DD_X)$ acts freely on $\Stab^\dagger(\DD_X)$.
So one does not need to take the quotient of the group of Calabi--Yau autoequivalences by $\Aut_{\mathrm{tri}}^0(\DD)$ as in the abelian surface case (c.f. \cite{Bri2}).
\end{Rem}

%

\subsection{Abelian variety}
Let $X$ be an abelian variety of dimension $n$. 
Since there is no quantum corrections and the Chern classes are trivial, 
the expected central charge at the complexified K\"ahler moduli $\omega \in H^2(X;\C)$ is given by 
$$
\ZZ_{\mho_{\omega}}(E)=-\langle\mho_{\omega}, v_X(E)\rangle_{\Muk}=-\int_X e^{- \omega}\ch(E).
$$
The existence of Bridgeland stability condition with this central charge is known for $n\le 3$. 
By Proposition \ref{WP_B}, the Weil--Petersson potential is
$$
K_{\WP}(\tau)=-\log(\Im(\omega)^n) - \log \frac{2^n}{n!}.
$$

Fix a polarization $H$. We think of $\omega=\tau H$ for $\tau \in \HH$ as a slice of the stringy K\"ahler moduli space $\mathcal{M}_{\mathrm{Kah}}(X)$.  
Then the Weil--Petersson metric on $\HH$ is essentially the Poincar\'e metric.  
This example is a toy model in the sense that there is no quantum correction. 

The above observation is compatible with Wang's {\it mirror} result \cite[Remark 1.3]{Wang}, 
which says that in the case of infinite distance,
the Weil--Petersson metric is asymptotic to a scaling of the Poincar\'e metric.

\subsection{Quintic threefold}
Although the existence of a Bridgeland stability condition for a quintic threefold $X \subset \mathbb{P}^4$ has not yet been proven, 
we can still compute the Weil--Petersson potential using the central charge in Equation (\ref{Central charge}) near the large volume limit. 

Let $\tau H \in H^2(X;\C)$ be the complexified K\"ahler class, where $H$ is the hyperplane class and $\tau \in \HH$. 
First we observe that 
$$
\exp_*(\tau H)=1+\tau H+\frac{\tau^2}{2}(1+\frac{1}{5}\sum_{d\geq1}N_d d^3 q^d)H^2
+\frac{\tau^3}{6}(1+\frac{1}{5}\sum_{d\geq1}N_d d^3 q^d)H^3, 
$$
where $q=e^{2\pi  \sqrt{-1} \tau}$ and $N_d^X$ denotes the genus 0 Gromov--Witten invariant of $X$ of degree $d$, 
and we use the quantum product 
$$
H*H=\Phi(q)H^2=\frac{1}{5}(5+\sum_{d\geq1}N_d^Xq^d d^3)H^2. 
$$
Then the central charge computes to be 
\begin{align}
\ZZ(E)&=- \left \langle \exp_*( \tau H  ), v_X(E) \right\rangle_{\Muk} \notag \\
&=-\int_X e^{-\tau H}v_X(E)
+\frac{\zeta(3) \chi(X)}{(2\pi)^3}(\frac{\tau^2}{10}H^2\ch_1(E)-\frac{\tau^3}{6}\ch_0(E)) \sum_{d\geq1}N^X_d d^3 q^d, \notag 
\end{align}
where $\chi(X)$ is the topological Euler number of $X$. 
Near the large volume limit, the Weil--Petersson potential is given by
\begin{align}
K_{\WP}(\tau) &= -\log\Big(H^3(\overline{\Phi(q)}(\frac{\bar{\tau}^3}{6}+\frac{\tau \bar{\tau}^2}{2})-\Phi(q)(\frac{{\tau^3}}{6}+\frac{\tau^2 \bar{\tau}}{2}) \Big)-2\log\Big(\frac{\zeta(3) \chi(X)}{(2\pi)^3}\Big) \notag \\
& \sim -\log(\frac{4}{3}H^3 \Im(\tau)^3) -2\log\Big(\frac{\zeta(3) \chi(X)}{(2\pi)^3}\Big) +O(q). \notag
\end{align}
Therefore the Weil--Petersson metric of a quintic threefold is a quantum deformation of the Poincar\'e metric on $\HH$ as expected. 
In particular, for sufficiently small $q$, it is non-degenerate and the Weil--Petersson distance to the large volume limit is infinite. 
When there is no B-field, i.e. $\tau \in \sqrt{-1} \R$, the correction term $O(q)$ is explicitly given by $\log(\Phi(q))$.

\begin{Rem}[\cite{COGP}]
The stringy K\"ahler moduli space $\mathcal{M}_{\mathrm{Kah}}(X)$ of a quintic Calabi--Yau threefold $X\subset \mathbb{P}^4$ is expected to be identified with the suborbifold
$$
[\{ z \in \C \ | \ z^5\ne1\}/\Z_5] \subset [\mathbb{P}^1/\Z_5]. 
$$ 
The point $z=\infty$ is the large volume limit, the point $z^5=1$ is the conifold point, and the point $z=0$ is the Gepner point. 
We expect the following properties of the Weil--Petersson metric on $\mathcal{M}_{\mathrm{Kah}}(X)$.
	\begin{enumerate}
	\item The Weil--Petersson distance to the conifold point, which corresponds to a quintic threefold with a conifold singularity, should be finite. 
This is based on a result of Wang \cite{Wang} on the mirror complex moduli, 
which asserts that if a Calabi--Yau variety has at worst canonical singularities, then it has finite Weil--Petersson metric along any smoothing to Calabi--Yau manifolds.
	\item The Weil--Petersson metric at the Gepner point should be an orbifold metric. 
	This is because the auto-equivalence 
	$$
	\Phi(-)=\mathrm{ST}_{\mathcal{O}_X} \circ \big( (-) \otimes\mathcal{O}_X(H) \big), 
	$$
	where $\mathrm{ST}_{\mathcal{O}_X}$ denotes the Seidel--Thomas spherical twist with respect to $\mathcal{O}_X$, 
	at the Gepner point satisfies the relation $\Phi^5=[2]$. 
	This descends to $\Phi^5=\mathrm{id}$ on $K(\DD_X)$.  
	On other hand, the calculations of Candelas--de la Ossa--Green--Parkes \cite{COGP} shows that the Weil--Petersson curvature tends to $+\infty$ as we approach the Gepner point. 
	\end{enumerate}
\end{Rem}

It is interesting to investigate the interplay among the geometry of a Calabi--Yau threefold $X$, the cubic intersection form on $H^{2}(X;\Z)$, 
and curvature properties of the Weil--Petersson metric near a large volume limit \cite{TW, Wan, Wil2}. 

On the other hand, probably a more alluring research direction is to examine the Weil--Petersson metric away from a large volume limit, 
where central charges are not of the form (\ref{Central charge}), as the metric is inherently global. 
For instance, the Weil--Petersson metric around a Gepner point may be studied via matrix factorization categories via the Orlov equivalence \cite{Orl3}
$$
\mathrm{D^bCoh}(X) \cong \mathrm{HMF}(W),
$$
where $\mathrm{HMF}(W)$ is the homotopy category of a graded matrix factorization of the defining equation $W$ of the quintic 3-fold $X$.  
Toda studied stability conditions, called the Gepner type stability conditions, conjecturally corresponding to the Gepner point \cite{Toda}. 


\par\noindent{\scshape \small
Department of Mathematics, Harvard University\\
One Oxford street, Cambridge, MA 02138, USA}
\par\noindent{\ttfamily ywfan@math.harvard.edu}
\par\noindent{\ttfamily yau@math.harvard.edu} 
\ \\
\par\noindent{\scshape \small
Department of Mathematics, Kyoto University\\
Kitashirakawa-Oiwake, Sakyo, Kyoto, 606-8502, Japan}
\par\noindent{\ttfamily  akanazawa@math.kyoto-u.ac.jp}

\end{document}